\documentclass{article}
\usepackage{amsfonts}
\usepackage{amssymb}
\usepackage{amsmath}

\newtheorem{theorem}{Theorem}

\newtheorem{proposition}[theorem]{Proposition}

\newenvironment{proof}[1][Proof]{\noindent\textbf{#1.} }{\ \rule{0.5em}{0.5em}}
\input{tcilatex}
\begin{document}

\begin{center}
{\Large Bi-Hamiltonian Structure of Gradient Systems in Three \ }

{\Large Dimensions and Geometry of Potential Surfaces }

\bigskip

E. Abado\u{g}lu, \ 

Department of Mathematics, Yeditepe University,

34755 Ata\c{s}ehir, \.{I}stanbul, Turkey

eabadoglu@yeditepe.edu.tr\bigskip

H. G\"{u}mral\footnote{%
On leave of absence from Department of Mathematics, Yeditepe University, Ata%
\c{s}ehir, Istanbul.}

Department of Mathematics, Australian College of Kuwait,

13015, West Mishref, Kuwait
\end{center}

\bigskip

\textbf{Abstract: }

Working bi-Hamiltonian structure and Jacobi identity in Frenet-Serret frame
associated to a dynamical system, we proved that all dynamical systems in
three dimensions possess two compatible Poisson structures. We investigate
relations between geometry of surfaces defined by potential function of a
gradient system and its bi-Hamiltonian structure. We show that it is
possible to find Hamiltonian functions whose gradient flows have geodesic
curvature zero on potential surfaces. Using this, we conclude that
Hamiltonian functions are determined by distance functions on potential
surfaces. We apply this technique to find conserved quantities of three
dimensional gradient systems including the Aristotelian model of the
three-body motion.

\section{Introduction}

A Poisson structure on a manifold is defined by a skew symmetric
contravariant bilinear form subjected to the Jacobi identity expressed as
the vanishing of the Schouten bracket of Poisson tensor with itself \cite%
{lich77}-\cite{mr94}. This structure having no non-degeneracy requirement
becomes the basic underlying geometry to study non-canonical Hamilton's
equations on odd dimensional manifolds as well as the Hamiltonian structures
of nonlinear evolution equations \cite{olver}-\cite{morrison}.

The first interesting case of a completely degenerate finite dimensional
Hamiltonian structure occurs in three dimensions. Many works have been
devoted to the study of three dimensional dynamical systems with primary
concern on quantization, construction of conserved quantities, Hamiltonian
structures, integrability problems and their numerical integration using
techniques from various areas such as Poisson geometry, differential
equations, Frobenius integrability theorem and theory of foliations \cite%
{nambu}-\cite{ben}.

In \cite{eh09}, we reduced the problem of constructing Hamiltonian
structures in three dimensions to the solutions of a Riccati equation in
moving coordinates of Frenet-Serret frame. All known examples of dynamical
systems having two compatible and explicit Hamiltonian structures are
exhausted by constant solution. We concluded that in three dimensions vector
fields which are not eigenvectors of the curl operator are at least locally
bi-Hamiltonian. This structure manifests itself in the well-known form of
Frenet-Serret triad $\mathbf{t}=\mathbf{n}\times \mathbf{b}$ where $\mathbf{t%
}$ is the unit tangent vector associated with the given dynamical system
and, the normal vectors $\mathbf{n}$ and $\mathbf{b}$ are related to
conserved covariants. From the expression of gradient operator in
Frenet-Serret frame (c.f. Eq.(\ref{e17})) we observe that unit tangent
vector arises from the Cartesian gradient of a function of arclenght
variable.

In this work, we shall continue to investigate the local structure of
bi-Hamiltonian systems in three dimensions assuming that they are described
by gradient vector fields. We shall relate ingredients of bi-Hamiltonian
structure to geometry of surfaces described by the potential function $F$.
In particular, we shall prove that there are two Hamiltonian functions
related with geodesic coordinates on the potential surfaces.

\subsection{Content of the work}

In the next section, we shall first develop differential calculus in
Frenet-Serret frame that will be used throughout the paper. Then, we shall
survey on conditions for construction of Frenet-Serret frame for a given
vector field in $%
\mathbb{R}
^{3}$ to ensure its existence. This will extend our previous result in \cite%
{eh09} that excludes eigenvectors of curl operator.

In Section 3, we shall summarize ingredients of bi-Hamiltonian systems in
three dimensions. In particular, we shall identify a Poisson bi-vector with
a locally integrable (in the sense of Frobenius) vector field in three
dimensions and, will work with the latter. We shall first reduce the Jacobi
identity into Riccati equation and then, assuming we have independent
solutions, we shall exhibit relations between conserved Hamiltonians and
Poisson vectors.

In Section 4, we shall start with geometric characterization of potential
surfaces in the normal coordinates of the Frenet-Serret frame. By
considering gradient flows of Hamiltonian functions on potential surfaces,
we shall prove that it is possible to find Hamiltonian functions whose
gradient flows on the potential surface have geodesic curvature zero. Using
this result, we shall conclude that Hamiltonian functions are determined by
distance functions on potential surfaces. Finally, we shall prove that
Hamiltonian functions of the gradient system are related with geodesic
coordinates of the potential surfaces.

In Sections 5, we shall present examples of gradient dynamical systems which
are bi-Hamiltonian and, work out in details the geometry of potential
surface.

\section{Frenet-Serret Frame}

Let $\mathbf{t(}x,y,z\mathbf{)}$ be a given unit vector field in $%
\mathbb{R}
^{3}$ endowed with Cartesian coordinates $\mathbf{x=(}x,y,z\mathbf{)}$. We
may assume $\mathbf{t}$ to be a unit tangent vector to a curve $t\rightarrow 
\mathbf{x}(t)$ in $%
\mathbb{R}
^{3}$. This may be the solution of dynamical system in Eq.(\ref{e1}). In
this case, we can choose $\mathbf{t}=\mathbf{v}/||\mathbf{v}||$. Locally,
one can always lift $\mathbf{t}$ to an orthonormal frame in $%
\mathbb{R}
^{3}$ in infinitely many ways. Let $\left( \mathbf{t},\mathbf{n},\mathbf{b}%
\right) $ be such an arbitrary orthonormal frame satisfying%
\begin{equation*}
\mathbf{t}=\mathbf{n}\times \mathbf{b},\text{ \ \ }\mathbf{n=b\times t},%
\text{ \ \ }\mathbf{b=t}\times \mathbf{n}.
\end{equation*}%
We introduce the directional derivatives along the triad $\left( \mathbf{t},%
\mathbf{n},\mathbf{b}\right) $ as 
\begin{equation}
\begin{array}{ccccc}
\partial _{s}=\mathbf{t}\cdot \nabla & \quad & \partial _{n}=\mathbf{n}\cdot
\nabla & \quad & \partial _{b}=\mathbf{b}\cdot \nabla%
\end{array}
\label{eq2}
\end{equation}%
so that the variables $(s,n,b)$ are the coordinates associated with the
Frenet-Serret frame. Assuming the Cartesian coordinates are functions $%
\mathbf{x=x}(s,n.b)$ of Frenet-Serret coordinates we find, using Eq.(\ref%
{eq2}) the Jacobian matrix%
\begin{equation*}
\frac{\partial (x,y,z)}{\partial (s,n,b)}=\left\vert 
\begin{array}{ccc}
\partial _{s}x & \partial _{n}x & \partial _{b}x \\ 
\partial _{s}y & \partial _{n}y & \partial _{b}y \\ 
\partial _{s}z & \partial _{n}z & \partial _{b}z%
\end{array}%
\right\vert =\left\vert 
\begin{array}{ccc}
\mathbf{t} & \mathbf{n} & \mathbf{b}%
\end{array}%
\right\vert
\end{equation*}%
whose determinant 
\begin{equation*}
\det \left( \frac{\partial (x,y,z)}{\partial (s,n,b)}\right) =\mathbf{t\cdot 
}(\mathbf{n\times b})=1
\end{equation*}%
is non-zero and hence the inverse transformation%
\begin{equation}
s=s(\mathbf{x}),\text{ \ }n=n(\mathbf{x}),\text{ \ }b=b(\mathbf{x})
\label{coordinates}
\end{equation}%
exists locally, that is, in a sufficiently small neighborhood of a given
point $\mathbf{x}_{0}\in 
\mathbb{R}
^{3}$. These functions may be obtained by integrating the quantities%
\begin{equation*}
ds=\mathbf{t}\cdot d\mathbf{x},\text{ \ \ \ \ \ \ }dn=\mathbf{n}\cdot d%
\mathbf{x},\text{ \ \ \ \ \ \ }db=\mathbf{b}\cdot d\mathbf{x}
\end{equation*}%
the last two of which implies $n=$constant and $b=$constant when restricted
to the curve $\mathbf{x(}t\mathbf{)}$.

\subsection{Differential calculus}

By inverting equations (\ref{eq2}) we get the expression

\begin{equation}
\nabla =\mathbf{t}\partial _{s}+\mathbf{n}\partial _{n}+\mathbf{b}\partial
_{b}  \label{e17}
\end{equation}%
for the Cartesian gradient in Frenet-Serret frame. For future reference, we
define the helicities \cite{eh09}%
\begin{equation}
\mathcal{H}_{t}=\mathbf{t}\cdot \nabla \times \mathbf{t}\text{, \ \ }%
\mathcal{H}_{n}=\mathbf{n}\cdot \nabla \times \mathbf{n}\text{, \ \ }%
\mathcal{H}_{b}=\mathbf{b}\cdot \nabla \times \mathbf{b}  \label{hel}
\end{equation}%
and the cross-helicities%
\begin{eqnarray}
\mathcal{H}_{tn} &=&\mathbf{t}\cdot \nabla \times \mathbf{n}\text{, \ \ }%
\mathcal{H}_{nt}=\mathbf{n}\cdot \nabla \times \mathbf{t}\text{, \ \ }%
\mathcal{H}_{nb}=\mathbf{n}\cdot \nabla \times \mathbf{b}  \label{cross} \\
\mathcal{H}_{tb} &=&\mathbf{t}\cdot \nabla \times \mathbf{b}\text{, \ \ }%
\mathcal{H}_{bt}=\mathbf{b}\cdot \nabla \times \mathbf{t}\text{, \ \ }%
\mathcal{H}_{bn}=\mathbf{b}\cdot \nabla \times \mathbf{n.}  \notag
\end{eqnarray}%
which measure holonomicity of Frenet-Serret triad. From the coefficients of
the basis vectors in expansions of curls into $\left( \mathbf{t},\mathbf{n},%
\mathbf{b}\right) $ we obtain%
\begin{eqnarray}
\nabla \times \mathbf{t} &\mathbf{=}&\mathcal{H}_{t}\mathbf{t}+\mathcal{H}%
_{nt}\mathbf{n}+\mathcal{H}_{bt}\mathbf{b}  \notag \\
\nabla \times \mathbf{n} &\mathbf{=}&\mathcal{H}_{tn}\mathbf{t}+\mathcal{H}%
_{n}\mathbf{n}+\mathcal{H}_{bn}\mathbf{b}  \label{curls} \\
\nabla \times \mathbf{b} &\mathbf{=}&\mathcal{H}_{tb}\mathbf{t}+\mathcal{H}%
_{nb}\mathbf{n}+\mathcal{H}_{b}\mathbf{b}  \notag
\end{eqnarray}%
and from the orthonormality of basis vectors we have the divergences%
\begin{eqnarray}
\nabla \cdot \mathbf{t} &\mathbf{=}&\mathcal{H}_{bn}-\mathcal{H}_{nb}  \notag
\\
\nabla \cdot \mathbf{n} &\mathbf{=}&\mathcal{H}_{tb}-\mathcal{H}_{bt}
\label{divs} \\
\nabla \cdot \mathbf{b} &\mathbf{=}&\mathcal{H}_{nt}-\mathcal{H}_{tn}. 
\notag
\end{eqnarray}%
Using the vector identity 
\begin{equation}
\nabla (\mathbf{u}\cdot \mathbf{v})=(\mathbf{u}\cdot \nabla )\mathbf{v}+(%
\mathbf{v}\cdot \nabla )\mathbf{u+u}\times (\nabla \times \mathbf{v)+v}%
\times (\nabla \times \mathbf{u)}  \label{iddot}
\end{equation}%
we compute the derivative of $\mathbf{t}$ in direction of $\mathbf{t}$ to be 
\begin{equation*}
\partial _{s}\mathbf{t}\mathbf{=(\mathbf{t}\cdot \nabla )t=-t}\times (\nabla
\times \mathbf{t)=-}(\mathbf{n\times b})\times (\nabla \times \mathbf{t)=n}%
\mathcal{H}_{bt}-\mathbf{b}\mathcal{H}_{nt}\text{.}
\end{equation*}%
Repeating for the other directions we have the derivatives%
\begin{eqnarray}
\partial _{s}\mathbf{t} &=&\mathbf{n}\mathcal{H}_{bt}-\mathbf{b}\mathcal{H}%
_{nt}  \notag \\
\partial _{n}\mathbf{n} &=&\mathbf{b}\mathcal{H}_{tn}-\mathbf{t}\mathcal{H}%
_{bn}  \label{alongitself} \\
\partial _{b}\mathbf{b} &=&\mathbf{t}\mathcal{H}_{nb}-\mathbf{n}\mathcal{H}%
_{tb}  \notag
\end{eqnarray}%
of basis vectors along their respective directions. To find other
derivatives, add the identity 
\begin{equation*}
\nabla \times (\mathbf{u}\times \mathbf{v})=(\mathbf{v}\cdot \nabla )\mathbf{%
u}-(\mathbf{u}\cdot \nabla )\mathbf{v}+(\nabla \cdot \mathbf{v})\mathbf{u}%
-(\nabla \cdot \mathbf{u})\mathbf{v}
\end{equation*}%
to Eq.(\ref{iddot}) to obtain%
\begin{eqnarray*}
2(\mathbf{v}\cdot \nabla )\mathbf{u} &=&\nabla (\mathbf{u}\cdot \mathbf{v}%
)+\nabla \times (\mathbf{u}\times \mathbf{v})-(\nabla \cdot \mathbf{v})%
\mathbf{u}+(\nabla \cdot \mathbf{u})\mathbf{v} \\
&&\mathbf{-u}\times (\nabla \times \mathbf{v)-v}\times (\nabla \times 
\mathbf{u)}
\end{eqnarray*}%
and let $\mathbf{u}=\mathbf{t}$, $\mathbf{v}=\mathbf{n}$ to have%
\begin{eqnarray*}
2(\mathbf{n}\cdot \nabla )\mathbf{t} &=&2\mathbf{\partial }_{n}\mathbf{t} \\
&=&\nabla \times \mathbf{b}-(\nabla \cdot \mathbf{n})\mathbf{t}+(\nabla
\cdot \mathbf{t})\mathbf{n-t}\times (\nabla \times \mathbf{n)-n}\times
(\nabla \times \mathbf{t).}
\end{eqnarray*}%
Last two terms may be expressed as%
\begin{eqnarray*}
\mathbf{t}\times (\nabla \times \mathbf{n}) &=&(\mathbf{n\times b})\times
\left( \nabla \times \mathbf{n}\right) \\
&\mathbf{=}&\mathcal{H}_{n}\mathbf{b}-\mathcal{H}_{bn}\mathbf{n} \\
\mathbf{n}\times (\nabla \times \mathbf{t}) &=&(\mathbf{b\times t})\times
(\nabla \times \mathbf{t)} \\
&\mathbf{=}&\mathcal{H}_{bt}\mathbf{t}-\mathcal{H}_{t}\mathbf{b.}
\end{eqnarray*}%
Collecting these for $\mathbf{\partial }_{n}\mathbf{t}$ and repeating
similar computations for other derivatives we obtain%
\begin{eqnarray}
\mathbf{\partial }_{n}\mathbf{t} &=&\mathcal{H}_{bn}\mathbf{n}+\frac{1}{2}(%
\mathcal{H}_{t}-\mathcal{H}_{n}+\mathcal{H}_{b})\mathbf{b}  \notag \\
\mathbf{\partial }_{b}\mathbf{t} &=&-\frac{1}{2}(\mathcal{H}_{t}+\mathcal{H}%
_{n}-\mathcal{H}_{b})\mathbf{n}-\mathcal{H}_{nb}\mathbf{b}  \notag \\
\mathbf{\partial }_{t}\mathbf{n} &=&\mathbf{-}\mathcal{H}_{bt}\mathbf{t}+%
\frac{1}{2}(\mathcal{H}_{t}-\mathcal{H}_{n}-\mathcal{H}_{b})\mathbf{b}
\label{alongothers} \\
\mathbf{\partial }_{b}\mathbf{n} &=&\frac{1}{2}(\mathcal{H}_{t}+\mathcal{H}%
_{n}-\mathcal{H}_{b})\mathbf{t}+\mathcal{H}_{tb}\mathbf{b}  \notag \\
\mathbf{\partial }_{t}\mathbf{b} &=&\mathcal{H}_{nt}\mathbf{t}-\frac{1}{2}(%
\mathcal{H}_{t}-\mathcal{H}_{n}-\mathcal{H}_{b})\mathbf{n}  \notag \\
\mathbf{\partial }_{n}\mathbf{b} &=&-\frac{1}{2}(\mathcal{H}_{t}-\mathcal{H}%
_{n}+\mathcal{H}_{b})\mathbf{t}+\mathcal{H}_{tn}\mathbf{n.}  \notag
\end{eqnarray}

\subsection{Constructing Frenet-Serret frame}

The choice of orthonormal frame is determined by the choice of $\mathbf{n}$.
In \cite{eh09}, we introduced such a frame assuming that the unit tangent is
not an eigenvector of the curl operator. Here, we want to release this
assumption and prove the existence of an orthonormal frame for all smooth
dynamical systems in three dimensions. Our result will rely on the
eigenvalue problem 
\begin{equation}
\nabla \times \mathbf{t}=\lambda (\mathbf{x})\mathbf{t}  \label{eve}
\end{equation}%
for the curl operator. If $\mathbf{t}$ is not an eigenvector of the curl
operator we have $\left( \nabla \times \mathbf{t}\right) \times \mathbf{t}%
\neq \mathbf{0}$ and we recover the result of \cite{eh09}. If however, the
eigenvalue equation (\ref{eve}) holds, then $\mathcal{H}_{t}=\mathbf{t}\cdot
\nabla \times \mathbf{t=}\lambda (\mathbf{x})$. At each point $\mathbf{x}$,
the eigenvalue $\lambda (\mathbf{x})$ will define a surface with normal $%
\nabla \mathcal{H}_{t}(\mathbf{x})$ if $\mathcal{H}_{t}(\mathbf{x})$ is not
a constant function. We distinguish two cases depending on whether the unit
tangent has components lying on this eigensurface or not. If it has, then we
choose the normal on the eigensurface. If $\mathbf{t}$ is completely aligned
with the surface normal, then we recall a result of Chandrasekhar and
Kendall in \cite{CK} that there exist a constant unit vector defining an
eigenvector of the curl operator and construct a frame with this constant
unit vector. In the remaining case with $\lambda (\mathbf{x})=0$, we have a
surface whose gradient is the unit tangent and we choose the frame using
lines of curvature of this surface. More precisely, following result proves
that there are canonical liftings to Frenet-Serret frames.

\begin{proposition}
Given a nonzero vector field $\mathbf{v}\in 
\mathbb{R}
^{3},$ let $\mathbf{t}=\mathbf{v}/||\mathbf{v}||$. Then, the vector field $%
\mathbf{n}$ can be chosen as follows

\textbf{1.} If $\left( \nabla \times \mathbf{t}\right) \times \mathbf{t}\neq 
\mathbf{0}$ then let 
\begin{equation}
\mathbf{n}=\frac{\left( \nabla \times \mathbf{t}\right) \times \mathbf{t}}{%
||\left( \nabla \times \mathbf{t}\right) \times \mathbf{t}||}  \label{e21a}
\end{equation}%
and we have necessarily $\mathcal{H}_{nt}=\mathbf{n}\cdot \nabla \times 
\mathbf{t}=0$.

\textbf{2.} If $\left( \nabla \times \mathbf{t}\right) \times \mathbf{t}=%
\mathbf{0}$, then $\nabla \times \mathbf{t}\mathbf{=}\mathcal{H}_{t}\mathbf{t%
}$ and we have necessarily $\mathcal{H}_{nt}=\mathcal{H}_{bt}=0$, or
equivalently $\partial _{s}\mathbf{t}=\mathbf{0}$. We distinguish two cases:

\textbf{2a.} if $\ \nabla \mathcal{H}_{t}\times \mathbf{t}\neq \mathbf{0}$
then choose 
\begin{equation}
\mathbf{n}=\frac{\nabla \mathcal{H}_{t}\times \mathbf{t}}{\mathbf{||}\nabla 
\mathcal{H}_{t}\times \mathbf{\mathbf{t}||}}\mathbf{.}  \label{e21b}
\end{equation}

\textbf{2bi.} if $\nabla \mathcal{H}_{t}\times \mathbf{t}=\mathbf{0}$ and $%
\mathcal{H}_{t}=$ constant $\neq 0$\textbf{,} then there exists a constant
unit vector $\mathbf{a}$ such that%
\begin{equation}
\mathbf{n}=\frac{\mathbf{a}\times \mathbf{t}}{||\mathbf{a}\times \mathbf{t}||%
}  \label{e21c}
\end{equation}%
and hence $\mathcal{H}_{nt}=\mathcal{H}_{bt}=0$.

\textbf{2bii. }if $\nabla \mathcal{H}_{t}\times \mathbf{t}=\mathbf{0}$ and $%
\mathcal{H}_{t}=$ $0$, then there exists a surface with normal $\mathbf{t}$
and the Frenet-Serret frame is the Darboux frame of the lines of curvature
on this surface. We have $\mathcal{H}_{nt}=\mathcal{H}_{bt}=0.$
\end{proposition}

\begin{proof}
Case 1. follows easily. For case 2., assume $\left( \nabla \times \mathbf{t}%
\right) \times \mathbf{t}=\mathbf{0}$. Then, $\nabla \times \mathbf{t}$ is
proportional to $\mathbf{t}$, $\nabla \times \mathbf{t}\mathbf{=}\mathcal{H}%
_{t}\mathbf{t}$ and we have $\mathcal{H}_{nt}=\mathcal{H}_{bt}=0$ which are
the coefficients of unit vectors in the expression for $\partial _{s}\mathbf{%
t}$. In order to construct the normal vector $\mathbf{n}$ we have two
subcases depending on whether $\nabla \mathcal{H}_{t}$ is non-zero and
parallel to $\mathbf{t}$. If $\nabla \mathcal{H}_{t}\times \mathbf{t}\neq 
\mathbf{0}$ then define the unit normal as in Eq.(\ref{e21b}). If $\nabla 
\mathcal{H}_{t}\times \mathbf{t}=\mathbf{0}$ we have $\nabla \mathcal{H}_{t}$
proportional to $\mathbf{t}$. If the proportionality function is zero, we
have $\mathcal{H}_{t}=$constant. In this case, $\mathbf{t}$ is an
eigenvector of the curl operator with constant eigenvalue $\mathcal{H}_{t}$.
By a result obtained in \cite{CK} there exists a scalar function $\psi $
satisfying 
\begin{equation}
\bigtriangleup \psi +\mathcal{H}_{t}^{2}\psi =0  \label{e30}
\end{equation}%
and a constant vector $\mathbf{a}$ such that the eigenvector $\mathbf{t}$
can be expressed as 
\begin{equation}
\mathbf{t}=\frac{1}{\mathcal{H}_{t}}\nabla \times \left( \psi \mathbf{a}%
+\nabla \times \left( \psi \mathbf{a}\right) \right) .  \label{e31}
\end{equation}%
Then, we define the normal vector by Eq. (\ref{e21c}). It follows that 
\begin{eqnarray*}
\mathcal{H}_{nt} &=&\mathbf{n}\cdot \nabla \times \mathbf{t}=\frac{\mathbf{a}%
\times \mathbf{t}}{||\mathbf{a}\times \mathbf{t}||}\cdot \mathcal{H}_{t}%
\mathbf{t}=0 \\
\mathcal{H}_{bt} &=&\mathbf{b}\cdot \nabla \times \mathbf{t}=\mathbf{(t}%
\times \mathbf{n)}\cdot \mathcal{H}_{t}\mathbf{t}=0.
\end{eqnarray*}%
If the proportionality function is non-zero, taking curl and then dot
product with $\mathbf{t}$ we get $\mathcal{H}_{t}=0$ which is the
integrability condition for the unit tangent vector. Since, this is a
subcase of 2 with $\nabla \times \mathbf{t}=\mathcal{H}_{t}\mathbf{t}$, we
have $\nabla \times \mathbf{t}=\mathbf{0}$. Locally, there exists a function 
$F\left( \mathbf{x}\right) $ such that $\mathbf{t}=\nabla F\left( \mathbf{x}%
\right) $. Choose $\mathbf{n}$ to be the unit tangent vector of the line of
curvature of this surface, namely, a vector $\mathbf{n}$ satisfying 
\begin{equation}
\mathbf{n}\cdot \left( \mathbf{t}\times \left( \mathbf{n}\cdot \nabla
\right) \mathbf{t}\right) =0
\end{equation}%
which implies, from Eq.(\ref{alongothers}) $\mathcal{H}_{t}=0$, $\mathcal{H}%
_{n}=\mathcal{H}_{b}$. Moreover, we have $\mathcal{H}_{nt}=\mathcal{H}%
_{bt}=0 $ as in the previous case.
\end{proof}

\section{Bi-Hamiltonian Structure}

We shall summarize the necessary ingredients of the bi-Hamiltonian formalism
in three dimensions. See \cite{nambu}-\cite{ben} for details and examples.
For $\mathbf{x=}\left\{ x^{i}\right\} =(x,y,z)\in 
\mathbb{R}
^{3}$, $t\in 
\mathbb{R}
$ and overdot denoting the derivative with respect to $t$, we consider the
system of autonomous differential equations%
\begin{equation}
\frac{d\mathbf{x}}{dt}=\mathbf{v}\left( \mathbf{x}\right)  \label{e1}
\end{equation}%
associated with a three-dimensional smooth vector field $\mathbf{v}.$ Eq.(%
\ref{e1}) is said to be Hamiltonian if the right hand side can be written as 
$\mathbf{v}\left( \mathbf{x}\right) =\Omega \left( \mathbf{x}\right) \left(
dH\left( \mathbf{x}\right) \right) $ where $H\left( \mathbf{x}\right) $ is
the Hamiltonian function and $\Omega \left( \mathbf{x}\right) $ is the
Poisson bi-vector (i.e. a skew-symmetric, contravariant two-tensor)
subjected to the Jacobi identity $\left[ \Omega \left( \mathbf{x}\right)
,\Omega \left( \mathbf{x}\right) \right] =0$ defined by the Schouten bracket 
\cite{lich77}. In coordinates, if $\partial _{i}=\partial /\partial x^{i}$,
the Poisson bi-vector is $\Omega \left( \mathbf{x}\right) =\Omega
^{jk}\left( \mathbf{x}\right) \partial _{j}\wedge \partial _{k}$, with
summation over repeated indices, and the Jacobi identity reads $\Omega
^{i[j}\partial _{i}\Omega ^{kl]}=0$ where $[jkl]$ denotes the
antisymmetrization over three indices. It follows that in three dimensions
the Jacobi identity is a single scalar equation. One can exploit vector
calculus and differential forms in three dimensions to have a more
transparent understanding of Hamilton`s equations as well as the Jacobi
identity.\ Using the isomorphism 
\begin{equation}
J_{i}=\varepsilon _{ijk}\Omega ^{jk}\qquad i,j,k=1,2,3  \label{e4}
\end{equation}%
between skew-symmetric matrices and (pseudo)-vectors defined by the
completely antisymmetric Levi-Civita tensor $\varepsilon _{ijk}$, we can
write the Hamilton's equations and the Jacobi identity as 
\begin{equation}
\mathbf{v}=\mathbf{J}\times \nabla H\text{ \ \ \ \ \ \ \ }\mathbf{J}\cdot
\left( \nabla \times \mathbf{J}\right) =0,  \label{e5}
\end{equation}%
respectively. In this form the Jacobi identity is equivalent to the
Frobenius integrability condition $J\wedge dJ=0$ for the one form $%
J=J_{i}dx^{i}$. It is the condition for $J$ to define a foliation of
codimension one in three dimensional space \cite{hasan},\cite{arn}-\cite%
{reinhart}. A distinguished property of Poisson structures in three
dimensions is the invariance of the Jacobi identity under the multiplication
of Poisson vector $\mathbf{J}\left( \mathbf{x}\right) $ by an arbitrary but
non-zero factor. The identities

\begin{equation}
\mathbf{J}\cdot \mathbf{v}=0,\text{ \ \ \ }\nabla H\cdot \mathbf{v}=0
\label{e11}
\end{equation}%
follows directly from the Hamilton's equations in $\left( \ref{e5}\right) $.
The second equation in (\ref{e11}) is the expression for the conservation of
Hamiltonian function. A three dimensional vector field $\mathbf{v}\left( 
\mathbf{x}\right) $ is said to be bi-Hamiltonian if there exist two
different compatible Hamiltonian structures \cite{olver},\cite{magri78}. In
the notation of equation $\left( \ref{e5}\right) $, this implies%
\begin{equation}
\mathbf{v}=\mathbf{J}_{1}\times \nabla H_{2}=\mathbf{J}_{2}\times \nabla
H_{1}  \label{e11a}
\end{equation}%
for the dynamical equations. The compatibility condition for $\mathbf{J}_{1}$
and $\mathbf{J}_{2}$ is defined by the Jacobi identity for the Poisson
pencil $\mathbf{J}_{1}+c\mathbf{J}_{2}$ for arbitrary constant $c.$

\subsection{Jacobi identity in Frenet-Serret frame}

Given a vector field $\mathbf{v}$, $\mathbf{v}/\left\Vert \mathbf{v}%
\right\Vert $ is the unit tangent vector $\mathbf{t}$ to the flow of $%
\mathbf{v}$. It follows from the identity $\mathbf{J}\cdot \mathbf{v}=0$
that the Poisson vector $\mathbf{J}$ has no component along the unit tangent
vector $\mathbf{t}.$ Hence, we set 
\begin{equation}
\mathbf{J}=A\mathbf{n}+B\mathbf{b}  \label{e26}
\end{equation}%
for unknown functions $A\left( \mathbf{x}\right) $ and $B\left( \mathbf{x}%
\right) $ satisfying $A^{2}+B^{2}\neq 0$. Assuming $A\neq 0$ and defining
the function $\mu =B/A$ the Jacobi identity for $\mathbf{J=}A(\mathbf{n}+\mu 
\mathbf{b)}$ reduces to the Riccati equation%
\begin{equation}
\partial _{s}\mu =\mathcal{H}_{n}+\mu (\mathcal{H}_{nb}+\mathcal{H}%
_{bn})+\mu ^{2}\mathcal{H}_{b}  \label{ricca}
\end{equation}%
and 
\begin{equation}
\partial _{s}\ln A=\partial _{s}\ln \left\Vert \mathbf{v}\right\Vert -\mu 
\mathcal{H}_{b}-\mathcal{H}_{nb}  \label{factor}
\end{equation}%
in the arclenght variable $s$. The Riccati equation (\ref{ricca}) is
equivalent to a linear second order equation and hence possesses two
linearly independent solutions leading to two Poisson vectors for dynamical
system under consideration. The Hamiltonian form of dynamical equations
implies that the Poisson vectors obtained from solutions of Riccati equation
are always compatible. Thus, we conclude that

\begin{proposition}
\label{mm}All dynamical systems in three dimensions possess two compatible
Poisson vectors.
\end{proposition}

\subsection{Poisson vectors and conserved quantities}

Once we have the independent solutions $\mu _{1}$ and $\mu _{2}$ of the
Riccati equation, we can form the compatible Poisson vectors 
\begin{equation}
\mathbf{J}_{1}\mathbf{=}A_{1}(\mathbf{n}+\mu _{1}\mathbf{b)}\text{, \ \ \ \
\ \ \ }\mathbf{J}_{2}\mathbf{=}A_{2}(\mathbf{n}+\mu _{2}\mathbf{b)}
\label{j1j2}
\end{equation}%
with conformal factors $A_{1}$ and $A_{2}$. The construction of
corresponding Hamiltonian functions to form a bi-Hamiltonian pair requires
integration of these Poisson vectors.

\begin{proposition}
The conserved covariants for $\mathbf{J}_{1}$ and $\mathbf{J}_{2}$ are%
\begin{equation}
\nabla H_{1}=\frac{||\mathbf{v}||}{A_{1}A_{2}(\mu _{2}-\mu _{1})}\mathbf{J}%
_{1}\text{, \ \ \ }\nabla H_{2}=\frac{-||\mathbf{v}||}{A_{1}A_{2}(\mu
_{2}-\mu _{1})}\mathbf{J}_{2},  \label{coco}
\end{equation}%
respectively.
\end{proposition}

\begin{proof}
By definition, the conserved Hamiltonians satisfy $\mathbf{v\cdot }\nabla
H_{1}=\mathbf{v\cdot }\nabla H_{1}=0$. That means, the gradients lie on the
space spanned by $\{\mathbf{n}_{1}\mathbf{,n}_{2}\}$ or, equivalently, by $\{%
\mathbf{J}_{1}\mathbf{,J}_{2}\}$. We, thus, form the linear combinations%
\begin{equation*}
\nabla H_{1}=a\mathbf{J}_{1}\mathbf{+}b\mathbf{J}_{2}\text{, \ \ }\nabla
H_{2}=c\mathbf{J}_{1}\mathbf{+}d\mathbf{J}_{2}
\end{equation*}%
and determine the coefficients. Bi-Hamiltonian form in Eq.(\ref{e11a})
together with the orthonormality of the basis $\{\mathbf{t},\mathbf{n,b}\}$
imply 
\begin{equation*}
b=-c=\frac{||\mathbf{v}||}{A_{1}A_{2}(\mu _{2}-\mu _{1})}.
\end{equation*}%
The identities $\nabla \times \nabla H_{1}=\nabla \times \nabla H_{2}\equiv
0 $ dotted with $\mathbf{J}_{1}$\textbf{\ }and $\mathbf{J}_{2}$ result,
recalling also the compatibility condition, in the fact that $a$ and $d$ are
also multiples of this function. It turns out that, for each Hamiltonian,
there is an equivalence class of Hamiltonian functions whose gradients
differ by the functions $a$ and $d$. Hence, without restriction to
generality, the conserved covariants are given by Eq.(\ref{coco}).
\end{proof}

\section{Gradient Systems}

We assume that there exists a potential function $F$ for the velocity field 
\begin{equation}
\frac{d\mathbf{x}}{dt}=\mathbf{v}\left( \mathbf{x}\right) =\nabla F\left( 
\mathbf{x}\right)  \label{gradsys}
\end{equation}%
of a given dynamical system in three dimensions. In this case, we have $%
\nabla \times \mathbf{v}\left( \mathbf{x}\right) =\mathbf{0}$, so that 
\begin{equation}
\nabla \times \mathbf{t}=\mathbf{t}\times \nabla \ln \left\Vert \mathbf{v}%
\right\Vert  \label{crost}
\end{equation}%
and hence, $\mathbf{t}\cdot \nabla \times \mathbf{t}=0$. Then, either $%
(\nabla \times \mathbf{t})\times \mathbf{t\neq 0}$ or $\nabla \times \mathbf{%
t=0.}$ In the first case we take $\mathbf{n}=((\nabla \times \mathbf{t}%
)\times \mathbf{t)/||}(\nabla \times \mathbf{t})\times \mathbf{t||}$ with $%
\mathbf{n}\cdot \nabla \times \mathbf{t}=0$ as well. Thus, the Frenet-Serret
frame for a gradient system with potential function $F$ may consists of unit
vector fields%
\begin{equation}
\mathbf{t}=\frac{\mathbf{v}}{||\mathbf{v||}}=\frac{\nabla F}{||\nabla F||}%
\text{, \ \ }\mathbf{n=}\frac{(\nabla \times \mathbf{t})\times \mathbf{t}}{%
\mathbf{||}\nabla \times \mathbf{t||}}\text{, \ \ }\mathbf{b}=\frac{\nabla
\times \mathbf{t}}{\mathbf{||}\nabla \times \mathbf{t||}}.  \label{fs}
\end{equation}%
In the second case, there exists a surface with normal $\mathbf{t,}$ which
is nothing but any potential surface of $F\left( \mathbf{x}\right) $ and the
Frenet-Serret frame is the Darboux frame of lines of curvature on the
potential surface. In other words, $\mathbf{t}$ will be unit normal of the
potential surfaces and choosing the Frenet-Serret frame as the Darboux frame
of the lines of curvature on the potential surface will work in both cases.

Now, since $\left\{ \mathbf{t,n,b}\right\} $ forms an orthonormal frame, the
arclengths of their integral curves $\left( s\mathbf{,}n,b\right) $ form a
local coordinate system around any point $p\in 
\mathbb{R}
^{3\text{ }}$provided that $\mathbf{v}\left( p\right) \neq \mathbf{0}.$ On
the other hand, 
\begin{equation}
||\mathbf{v}||\mathbf{t}=\nabla F\left( \mathbf{x}\right)  \label{potsur1}
\end{equation}%
implies that 
\begin{equation}
\begin{array}{ccccc}
\mathbf{t\cdot }\nabla F\left( \mathbf{x}\right) & = & \partial _{s}F\left(
s,n,b\right) & = & ||\mathbf{v||} \\ 
\mathbf{n\cdot }\nabla F\left( \mathbf{x}\right) & = & \partial _{n}F\left(
s,n,b\right) & = & 0 \\ 
\mathbf{b\cdot }\nabla F\left( \mathbf{x}\right) & = & \partial _{b}F\left(
s,n,b\right) & = & 0%
\end{array}
\label{potsur2}
\end{equation}%
Hence $F\left( \mathbf{x}\right) =F\left( s\right) $ in the coordinate
system defined above. Therefore, with the assumption that $\mathbf{v}\left(
p\right) =\nabla F\left( p\right) \neq \mathbf{0}$ the potential surfaces
are 
\begin{equation}
F\left( \mathbf{x}\right) =F\left( s\right) =c^{\prime }  \label{potsur}
\end{equation}%
which implies $s=c$. As $s$ is the arclength of integral curve of (\ref%
{gradsys}), its flow $\Phi _{\sigma }$ acts on the $s$ coordinate and
therefore on the potential surfaces by translation, i.e. $\Phi _{\sigma
}\left( s,n,b\right) =\left( s+\sigma ,n,b\right) $. If $H\left( \mathbf{x}%
\right) =H\left( s,n,b\right) $ is a Hamiltonian function for (\ref{gradsys}%
), then it will be invariant by this flow $H\left( s+\sigma ,n,b\right)
=H\left( s,n,b\right) ,$ which implies that $H$ is independent of $s,$ in
other words fixing $s=c$ and letting $\sigma =s$ lead to 
\begin{equation}
H\left( s,n,b\right) =H\left( s,n,b\right) \mid _{s=c}.  \label{extham}
\end{equation}

The invariance of Hamiltonians under the flow implies that it is sufficient
to determine the Hamiltonian function on any potential surface. Namely, a
Hamiltonian function can be reduced to a Hamiltonian function on potential
surfaces. However, the converse may not be so straightforward. For the
construction of a Hamiltonian function out of a function defined on
potential surfaces, or equivalently, to extend a gradient vector $\nabla H$
on the potential surface to a gradient vector on $%
\mathbb{R}
^{3},$ the vectors $\mathbf{t}$ and $\mathbf{t}\times \nabla H,$ which are
perpendicular to $\nabla H,$ must be tangent to a surface in $%
\mathbb{R}
^{3}.$ This condition can be written as 
\begin{equation}
\left[ \alpha \mathbf{t\cdot }\nabla \mathbf{,}\left( \beta \mathbf{t}\times
\nabla H\right) \cdot \nabla \right] =0  \label{extcon}
\end{equation}%
for some functions $\alpha $ and $\beta $ on $%
\mathbb{R}
^{3},$ where $\left[ \cdot ,\cdot \right] $ denotes the bracket of vector
fields. Our purpose is to prove that

\begin{theorem}
The Hamiltonian functions of a gradient system defined by a potential
function $F$ are determined by geodesic distance functions defined by
non-conjugate points on the potential surfaces.
\end{theorem}

To obtain this result, first we are going to show that, geodesic distances
on a potential surface are defined by a gradient system. Then, we will prove
that all Hamiltonian functions on potential surface are generated by
geodesic distances to non-conjugate points. Finally, we will show that these
distance functions can be extended to Hamiltonian functions on $%
\mathbb{R}
^{3}.$

\subsection{Differential geometry of potential surfaces}

We give geometric parameters of potential surface and of an arbitrary curve
on it.

\begin{proposition}
The fundamental forms of surfaces $F\left( s\right) =c$ are%
\begin{eqnarray*}
ds_{1}^{2} &=&dn^{2}+db^{2} \\
ds_{2}^{2} &=&\mathcal{H}_{bn}dn^{2}+(\mathcal{H}_{b}\mathbf{-}\mathcal{H}%
_{n})dbdn-\mathcal{H}_{nb}db^{2}.
\end{eqnarray*}%
If $\mathbf{X}\left( n\left( \sigma \right) ,b\left( \sigma \right) \right) $
is a curve, parametrized with the arclenght $\sigma $, on the potential
surface, then the normal and the geodesic curvatures are 
\begin{eqnarray}
\kappa _{n} &=&\xi ^{2}\mathcal{H}_{bn}\mathbf{+}\xi \eta (\mathcal{H}_{n}%
\mathbf{-}\mathcal{H}_{b})+\eta ^{2}\mathcal{H}_{nb}  \label{ncurv} \\
\kappa _{g} &=&\partial _{b}\xi -\partial _{n}\eta -\xi \mathcal{H}_{tn}%
\mathbf{-}\eta \mathcal{H}_{tb},  \label{gcurv}
\end{eqnarray}%
where we denote 
\begin{equation*}
dn/d\sigma =\xi ,\text{ \ \ }db/d\sigma =\eta .
\end{equation*}
\end{proposition}

\begin{proof}
Let $\mathbf{X}=\mathbf{X}(c,n,b)$ be a potential surface of the function $F$
determined by a constant $c$. The tangent vectors and the unit normal vector
of this potential surface are%
\begin{equation*}
\text{\ }\partial _{n}\mathbf{X}=\mathbf{n}\text{, \ \ \ }\partial _{b}%
\mathbf{X}=\mathbf{b}\text{, \ \ \ }\partial _{n}\mathbf{X}\times \text{\ }%
\partial _{b}\mathbf{X}=\mathbf{t.}
\end{equation*}%
The first fundamental form is obviously $ds_{1}^{2}=dn^{2}+db^{2}$. To find
the second fundamental form, first note that the normal vector $\mathbf{t}%
(c,n,b)$ of $\mathbf{X}(c,n,b)$ is independent of $s$. We compute%
\begin{eqnarray*}
d\mathbf{t\cdot }d\mathbf{X} &=&(\partial _{n}\mathbf{t}dn+\partial _{b}%
\mathbf{t}db)\cdot (\mathbf{n}dn+\mathbf{b}db) \\
&=&\mathbf{n}\cdot (\partial _{n}\mathbf{t}dn^{2}+\partial _{b}\mathbf{t}%
dbdn)+\mathbf{b}\cdot (\partial _{n}\mathbf{t}dndb+\partial _{b}\mathbf{t}%
db^{2}) \\
&=&(\mathbf{n}\cdot \partial _{b}\mathbf{t+\mathbf{b}\cdot \partial }_{n}%
\mathbf{\mathbf{t})}dbdn+\mathbf{n}\cdot \partial _{n}\mathbf{t}dn^{2}+%
\mathbf{b}\cdot \partial _{b}\mathbf{t}db^{2} \\
&=&\mathcal{H}_{bn}dn^{2}+(\mathcal{H}_{b}\mathbf{-}\mathcal{H}_{n})dbdn-%
\mathcal{H}_{nb}db^{2}
\end{eqnarray*}%
and thereby obtaining the second fundamental form. For curvatures, the unit
tangent vector of the curve $\mathbf{X}\left( n\left( \sigma \right)
,b\left( \sigma \right) \right) $ is 
\begin{equation}
\mathbf{T}=\frac{d\mathbf{X}}{d\sigma }=\frac{dn}{d\sigma }\mathbf{n}+\frac{%
db}{d\sigma }\mathbf{b}  \label{e147}
\end{equation}%
with $(dn/d\sigma )^{2}+(db/d\sigma )^{2}=1$. We compute the curvature vector%
\begin{eqnarray*}
\frac{d\mathbf{T}}{d\sigma } &=&\frac{d\xi }{d\sigma }\mathbf{n+}\xi \frac{d%
\mathbf{n}}{d\sigma }+\frac{d\eta }{d\sigma }\mathbf{b+}\eta \frac{d\mathbf{b%
}}{d\sigma }=\kappa \mathbf{N} \\
&=&(\xi \partial _{n}\xi +\eta \partial _{b}\xi )\mathbf{n}+\xi (\xi
\partial _{n}\mathbf{n}+\eta \partial _{b}\mathbf{n}) \\
&&+(\xi \partial _{n}\eta +\eta \partial _{b}\eta )\mathbf{b}+\eta (\xi
\partial _{n}\mathbf{b}+\eta \partial _{b}\mathbf{b}).
\end{eqnarray*}%
The second and the third parenthesis can be expressed, using formulas from
calculus in Frenet-Serret frame, as%
\begin{eqnarray*}
\xi \partial _{n}\mathbf{n}+\eta \partial _{b}\mathbf{n} &=&\xi (\mathcal{H}%
_{tn}\mathbf{b-}\mathcal{H}_{bn}\mathbf{t)}+\eta (-\mathcal{H}_{tn}\mathbf{n+%
}\frac{1}{2}(\mathcal{H}_{n}\mathbf{-}\mathcal{H}_{b})\mathbf{t),} \\
\xi \partial _{n}\mathbf{b}+\eta \partial _{b}\mathbf{b} &\mathbf{=}&\eta (-%
\mathcal{H}_{tb}\mathbf{n+}\mathcal{H}_{nb}\mathbf{t})+\xi (\mathcal{H}_{tb}%
\mathbf{b+}\frac{1}{2}(\mathcal{H}_{n}\mathbf{-}\mathcal{H}_{b})\mathbf{t),}
\end{eqnarray*}%
with which the curvature vector becomes 
\begin{eqnarray*}
\frac{d\mathbf{T}}{d\sigma } &=&(-\xi ^{2}\mathcal{H}_{bn}\mathbf{+}\xi \eta
(\mathcal{H}_{n}\mathbf{-}\mathcal{H}_{b})+\eta ^{2}\mathcal{H}_{nb}\mathbf{%
)t} \\
&&+(-\eta \partial _{n}\eta +\eta \partial _{b}\xi -\xi \eta \mathcal{H}_{tn}%
\mathbf{-}\eta ^{2}\mathcal{H}_{tb}\mathbf{)n} \\
&&+(\xi \partial _{n}\eta -\xi \partial _{b}\xi +\xi ^{2}\mathcal{H}_{tn}%
\mathbf{+}\xi \eta \mathcal{H}_{tb}\mathbf{)b} \\
&=&(-\xi ^{2}\mathcal{H}_{bn}\mathbf{+}\xi \eta (\mathcal{H}_{n}\mathbf{-}%
\mathcal{H}_{b})+\eta ^{2}\mathcal{H}_{nb}\mathbf{)t} \\
&&+(\partial _{b}\xi -\partial _{n}\eta -\xi \mathcal{H}_{tn}\mathbf{-}\eta 
\mathcal{H}_{tb}\mathbf{)(}\eta \mathbf{n-}\xi \mathbf{b})
\end{eqnarray*}%
where we used $\xi \partial _{n}\xi =-\eta \partial _{n}\eta $ and $\eta
\partial _{b}\eta =-\xi \partial _{b}\xi $. This gives the decomposition%
\begin{equation*}
\frac{d\mathbf{T}}{d\sigma }=\kappa _{n}\mathbf{t}+\kappa _{g}\mathbf{(}\eta 
\mathbf{n-}\xi \mathbf{b})
\end{equation*}%
of the curvature into tangential and geodesic components.
\end{proof}

A curve parametrized with arclength is a geodesic on a surface if $\kappa
_{g}=0$ for all points on the curve. We will see that this condition is
necessary for the integrability of forms forming transformations to geodesic
coordinates on the potential surface of a dynamical system. There will be
two such coordinate systems, each corresponds to construction of a
Hamiltonian function on the surface.

\subsection{Flows of Hamiltonian functions on potential surfaces and
geodesic distances}

In this section, we shall prove two propositions which relate the flows of
Hamiltonian functions on potential surfaces with geodesic distances. First
proposition tells that geodesics are integral curves of gradient flows of
geodesic distances on potential surfaces. Second proposition proves that
geodesic distances defined by non-conjugate points, i.e. points which are
joined by a unique geodesic, are functionally independent and therefore
defines two Hamiltonian functions. Hence these functions generates all
Hamiltonian functions that can be defined on the potential surface.

\begin{proposition}
Finding two Hamiltonian functions for the gradient system $\nabla F$ on the
potential surface $F=c$ amounts to finding geodesic distances on the
potential surface.
\end{proposition}

\begin{proof}
Let $(u,v)$ be orthogonal coordinates on the potential surface $F\left( 
\mathbf{x}\right) =c$ for a given Riemannian metric 
\begin{equation*}
\left( g_{ij}\left( u,v\right) \right) =\left( 
\begin{array}{cc}
g_{uu}\left( u,v\right) & 0 \\ 
0 & g_{vv}\left( u,v\right)%
\end{array}%
\right)
\end{equation*}%
and let $\left( q,p\right) $ be another orthogonal coordinate system such
that 
\begin{equation*}
\left( G_{ij}\left( q,p\right) \right) =\left( 
\begin{array}{cc}
G_{qq}\left( q,p\right) & 0 \\ 
0 & G_{pp}\left( q,p\right)%
\end{array}%
\right) .
\end{equation*}%
Above two metrics are related by%
\begin{eqnarray}
g_{uu}\left( u,v\right) &=&G_{qq}\left( q,p\right) q_{u}^{2}+G_{pp}\left(
q,p\right) p_{u}^{2}  \notag \\
g_{vv}\left( u,v\right) &=&G_{qq}\left( q,p\right) q_{v}^{2}+G_{pp}\left(
q,p\right) p_{v}^{2}  \label{e166} \\
0 &=&G_{qq}\left( q,p\right) q_{u}q_{v}+G_{pp}\left( q,p\right) p_{u}p_{v} 
\notag
\end{eqnarray}%
Note that, although $(u,v)$ and $\left( q,p\right) $ are orthogonal
coordinates on the surface, the coordinate transformation between them need
not be orthogonal, that is, the Jacobian determinant%
\begin{equation}
J=q_{u}p_{v}-q_{v}p_{u}=\sqrt{\frac{g_{uu}g_{vv}}{G_{qq}G_{pp}}}
\label{e167}
\end{equation}%
is not necessarily equal to unity. Defining 
\begin{equation}
\gamma _{u}^{q}=\sqrt{\frac{g_{uu}}{G_{qq}}}\text{, \ \ }\gamma _{v}^{q}=%
\sqrt{\frac{g_{vv}}{G_{qq}}}\text{, \ \ }\gamma _{u}^{p}=\sqrt{\frac{g_{uu}}{%
G_{pp}}}\text{, \ \ }\gamma _{v}^{p}=\sqrt{\frac{g_{vv}}{G_{pp}}}
\label{e167a}
\end{equation}%
and introducing a parameter $\theta $, we can write equations $\left( \ref%
{e166}\right) $ as a system of differential equations 
\begin{eqnarray}
\frac{\partial q}{\partial u} &=&\gamma _{u}^{q}\cos \theta \text{, \ \ \ \
\ \ }\frac{\partial p}{\partial u}=\gamma _{u}^{p}\sin \theta  \label{e167c}
\\
\frac{\partial q}{\partial v} &=&-\gamma _{v}^{q}\sin \theta \text{, \ \ \ \
\ }\frac{\partial p}{\partial v}=\gamma _{v}^{p}\cos \theta .  \label{e167b}
\end{eqnarray}%
for transformations between two orthogonal coordinates. Equations $\left( %
\ref{e167c}\right) $ and $\left( \ref{e167b}\right) $ are subjected to two
integrability conditions%
\begin{eqnarray}
\gamma _{u}^{q}\sin \theta \frac{\partial \theta }{\partial v}-\gamma
_{v}^{q}\cos \theta \frac{\partial \theta }{\partial u} &=&\frac{\partial
\gamma _{u}^{q}}{\partial v}\cos \theta +\frac{\partial \gamma _{v}^{q}}{%
\partial u}\sin \theta  \label{e168} \\
\gamma _{u}^{p}\cos \theta \frac{\partial \theta }{\partial v}+\gamma
_{v}^{p}\sin \theta \frac{\partial \theta }{\partial u} &=&\frac{\partial
\gamma _{v}^{p}}{\partial u}\cos \theta -\frac{\partial \gamma _{u}^{p}}{%
\partial v}\sin \theta .  \label{e169}
\end{eqnarray}%
The characteristic curve of the first integrability condition in Eq.$\left( %
\ref{e168}\right) $ is the integral curve of the dynamical system defined by%
\begin{eqnarray}
\frac{du}{dt} &=&-\gamma _{v}^{q}\cos \theta  \notag \\
\frac{dv}{dt} &=&\gamma _{u}^{q}\sin \theta  \label{e170} \\
\frac{d\theta }{dt} &=&\frac{\partial \gamma _{u}^{q}}{\partial v}\cos
\theta +\frac{\partial \gamma _{v}^{q}}{\partial u}\sin \theta  \notag
\end{eqnarray}%
and the arclength $\sigma _{1}$ of this integral curve satisfies%
\begin{equation*}
\frac{d\sigma _{1}}{dt}=\sqrt{g_{vv}}\gamma _{u}^{q}.
\end{equation*}%
In the transformed coordinates $(q,p),$ the first two equations for the
characteristic curve become%
\begin{equation}
\frac{dq}{dt}=-\frac{1}{\sqrt{G_{qq}}}\frac{d\sigma _{1}}{dt},\text{ \ \ \ }%
\frac{dp}{dt}=0  \label{e171a}
\end{equation}%
hence, the arclength $\rho _{1}$ in the transformed coordinates satisfies%
\begin{equation*}
\frac{d\rho _{1}}{dt}=\frac{d\sigma _{1}}{dt}
\end{equation*}%
and Eq.$\left( \ref{e171a}\right) $ in arclength parametrization becomes%
\begin{equation}
\frac{d}{d\rho _{1}}\left( 
\begin{array}{c}
q \\ 
p%
\end{array}%
\right) =-\sqrt{G_{qq}}\left( 
\begin{array}{c}
1 \\ 
0%
\end{array}%
\right) .  \label{e172}
\end{equation}%
The geodesic curvature of the characteristic curve $q=c$ is given by%
\begin{equation}
\kappa _{g}=\frac{1}{2\sqrt{G_{pp}}}\frac{\partial \ln G_{qq}}{\partial q}
\label{e172a}
\end{equation}

Similarly, for the second integrability condition $\left( \ref{e169}\right) $%
, the characteristic equation is the integral curve of the dynamical system%
\begin{eqnarray}
\frac{du}{dt} &=&\gamma _{v}^{p}\sin \theta  \notag \\
\frac{dv}{dt} &=&\gamma _{u}^{p}\cos \theta  \label{e173} \\
\frac{d\theta }{dt} &=&\frac{\partial \gamma _{v}^{p}}{\partial u}\cos
\theta -\frac{\partial \gamma _{u}^{p}}{\partial v}\sin \theta  \notag
\end{eqnarray}%
and the arclength $\sigma _{2}$ of this curve is determined by%
\begin{equation*}
\frac{d\sigma _{2}}{dt}=\sqrt{g_{uu}}\gamma _{v}^{p}.
\end{equation*}%
In the transformed coordinates $(q,p),$ the first two equations for
characteristic curve are%
\begin{equation}
\frac{dq}{dt}=0,\text{ \ \ \ }\frac{dp}{dt}=\frac{1}{\sqrt{G_{pp}}}\frac{%
d\sigma _{2}}{dt}  \label{e174a}
\end{equation}%
and the arclength $\rho _{2}$ in the transformed coordinates satisfies%
\begin{equation}
\frac{d\rho _{2}}{dt}=\frac{d\sigma _{2}}{dt}.  \label{e174b}
\end{equation}%
Equation $\left( \ref{e174a}\right) $ in this parametrization becomes%
\begin{equation}
\frac{d}{d\rho _{2}}\left( 
\begin{array}{c}
q \\ 
p%
\end{array}%
\right) =-\sqrt{G_{pp}}\left( 
\begin{array}{c}
0 \\ 
1%
\end{array}%
\right) .  \label{e175a}
\end{equation}%
The geodesic curvature of the characteristic curve $p=c$ is%
\begin{equation}
\kappa _{g}=\frac{1}{2\sqrt{G_{qq}}}\frac{\partial \ln G_{pp}}{\partial p}.
\label{e176}
\end{equation}

To arrive at the conclusion of the proposition we will restrict the
characteristic curves to be geodesics. To this end, we note that we can
choose two types of geodesic coordinates on a surface. One transforming the
metric as%
\begin{equation}
\left( 
\begin{array}{cc}
g_{uu}\left( u,v\right) & 0 \\ 
0 & g_{vv}\left( u,v\right)%
\end{array}%
\right) \longrightarrow \left( 
\begin{array}{cc}
1 & 0 \\ 
0 & G_{pp}\left( q,p\right)%
\end{array}%
\right)  \label{e177}
\end{equation}%
whereas, the other transforming the metric as 
\begin{equation}
\left( 
\begin{array}{cc}
g_{uu}\left( u,v\right) & 0 \\ 
0 & g_{vv}\left( u,v\right)%
\end{array}%
\right) \longrightarrow \left( 
\begin{array}{cc}
G_{qq}\left( q,p\right) & 0 \\ 
0 & 1%
\end{array}%
\right) .  \label{e178}
\end{equation}%
The first transformation reduces the equations (\ref{e172}) and (\ref{e172a}%
) to 
\begin{equation}
\frac{d}{d\rho _{1}}\left( 
\begin{array}{c}
q \\ 
p%
\end{array}%
\right) =-\left( 
\begin{array}{c}
1 \\ 
0%
\end{array}%
\right) ,\text{ \ \ \ }\kappa _{g}=0  \label{e179}
\end{equation}%
for the characteristic curve of the first integrability condition in Eq.(\ref%
{e168}) and its geodesic curvature. Similarly, the second transformation
reduces the equations (\ref{e175a}) and (\ref{e176}) to 
\begin{equation}
\frac{d}{d\rho _{2}}\left( 
\begin{array}{c}
q \\ 
p%
\end{array}%
\right) =-\left( 
\begin{array}{c}
0 \\ 
1%
\end{array}%
\right) ,\text{ \ \ \ }\kappa _{g}=0  \label{e181}
\end{equation}%
which are associated to characteristic curve of the second integrability
condition in Eq.(\ref{e169}).

From Eq.(\ref{e171a}) $\ $and Eq.(\ref{e174a}) we have 
\begin{equation*}
\frac{dq}{dt}=-\frac{d\sigma _{1}}{dt},\ \ \ \ \frac{dp}{dt}=\frac{d\sigma
_{2}}{dt}
\end{equation*}%
and without restriction of generality, we may assume that 
\begin{equation}
H_{1}=-q=\sigma _{1}\text{ \ \ \ }H_{2}=p=\sigma _{2}  \label{e183}
\end{equation}%
which are arclengths along geodesic curves, namely geodesic distances.
\end{proof}

Next, we are going to prove that geodesic distance functions defined by
non-conjugate points are two Hamiltonian functions on the potential surfaces.

\begin{proposition}
The Hamiltonian functions of a gradient system defined by a potential
function $F$ are determined by geodesic distance functions on potential
surfaces.
\end{proposition}

\begin{proof}
Since level surfaces, $F\left( \mathbf{x}\right) =c,$ are closed subspaces
of $%
\mathbb{R}
^{3}$, which is a complete metric space, level surfaces with the induced
metric are also complete. Then, by Hopf-Rinow theorem \cite{HR} they are
also geodesically complete, which implies the existence of a geodesic
between any two points on a level surface. Choose and fix a level surface $%
F\left( \mathbf{x}\right) =c,$ which we denote by $S_{c},$ and a point $%
\mathbf{p}_{1}\in S_{c}$. Then the first Hamiltonian function on the
potential surface, $H_{1}\left( \mathbf{x}\right) ,$ can be constructed as 
\begin{equation*}
H_{1}\left( \mathbf{x}\right) =d\left( \mathbf{x},\mathbf{p}_{1}\right) =%
\text{The length of the geodesic joining }\mathbf{x\in }S_{c}\text{ to }%
\mathbf{p}_{1}.\text{ }
\end{equation*}%
Since the gradients of distance functions have unit norm, 
\begin{equation*}
\left\Vert \nabla H_{1}\left( \mathbf{x}\right) \right\Vert =\left\Vert
\nabla d\left( \mathbf{x},\mathbf{p}_{1}\right) \right\Vert =1\text{ for }%
\mathbf{x\in }S_{c}
\end{equation*}%
and belong to tangent plane of $S_{c}$ at $\mathbf{x},$ it is a Hamiltonian
function of $\left( \text{\ref{gradsys}}\right) $, whose gradient has unit
norm on the potential surface.

For the second Hamiltonian, let $\Gamma \left( \mathbf{p}_{1}\right) $ be
the set of intersection of all geodesics through $\mathbf{p}_{1}.$ Choosing
another point $\mathbf{p}_{2}$ $\notin \Gamma $ on the level surface $S_{c},$
and repeating the same construction, the second Hamiltonian function is
obtained. This second function will be independent from the first one. In
the rest of the proof we will show this. Assume, on the contrary, that the
second distance function is not independent. Then, 
\begin{equation*}
\nabla d\left( \mathbf{x},\mathbf{p}_{1}\right) \times \nabla d\left( 
\mathbf{x},\mathbf{p}_{2}\right) =0
\end{equation*}%
which implies that 
\begin{equation*}
\nabla d\left( \mathbf{x},\mathbf{p}_{2}\right) =\lambda \nabla d\left( 
\mathbf{x},\mathbf{p}_{1}\right)
\end{equation*}%
and since they have unit norm 
\begin{equation*}
\left\Vert \nabla d\left( \mathbf{x},\mathbf{p}_{1}\right) \right\Vert
=\left\Vert \nabla d\left( \mathbf{x},\mathbf{p}_{2}\right) \right\Vert =1
\end{equation*}%
we obtain $\lambda =\pm 1.$ Now, if $\lambda =1$ then, 
\begin{equation*}
\nabla d\left( \mathbf{x},\mathbf{p}_{1}\right) =\nabla d\left( \mathbf{x},%
\mathbf{p}_{2}\right)
\end{equation*}%
and therefore%
\begin{equation*}
d\left( \mathbf{x},\mathbf{p}_{2}\right) =d\left( \mathbf{x},\mathbf{p}%
_{1}\right) +k.
\end{equation*}%
Choosing $\mathbf{x}$ as the midpoint of the geodesic combining $\mathbf{p}%
_{1}$ and $\mathbf{p}_{2}$ implies that $k=0$ on the surface, and therefore $%
\mathbf{p}_{2}=$ $\mathbf{p}_{1}\in \Gamma \left( \mathbf{p}_{1}\right) $
which contradicts our assumption. On the other hand if $\lambda =-1,$ then 
\begin{equation*}
\nabla d\left( \mathbf{x},\mathbf{p}_{1}\right) =-\nabla d\left( \mathbf{x},%
\mathbf{p}_{2}\right)
\end{equation*}%
and therefore%
\begin{equation*}
d\left( \mathbf{x},\mathbf{p}_{1}\right) +d\left( \mathbf{x},\mathbf{p}%
_{2}\right) =k
\end{equation*}%
Setting $\mathbf{x}=\mathbf{p}_{1}$ implies that $k=d\left( \mathbf{p}_{1},%
\mathbf{p}_{2}\right) ,$ and the condition becomes 
\begin{equation*}
d\left( \mathbf{x},\mathbf{p}_{1}\right) +d\left( \mathbf{x},\mathbf{p}%
_{2}\right) =d\left( \mathbf{p}_{1},\mathbf{p}_{2}\right) .
\end{equation*}%
By triangle inequality, this means that the points $\mathbf{x},\mathbf{p}%
_{1},\mathbf{p}_{2}$ lie on the same geodesic for all $\mathbf{x}$ on the
level surface. Hence, $\nabla d\left( \mathbf{x},\mathbf{p}_{1}\right)
=-\nabla d\left( \mathbf{x},\mathbf{p}_{2}\right) $ is possible only if
every geodesic through $\mathbf{p}_{1}$ contains $\mathbf{p}_{2},$ and
therefore $\mathbf{p}_{2}\in \Gamma \left( \mathbf{p}_{1}\right) $ which
again contradicts to our assumption. Namely, distance functions defined by
these two points specified above are functionally independent. Since the
gradient of these two functions span the tangent plane at any point,
gradient of any third function obtained in this way will be linearly
dependent on the previous two, therefore it will be a function of $H_{1}$
and $H_{2}$.
\end{proof}

To sum up, if we take $\mathbf{X}\left( n\left( \sigma _{i}\right) ,b\left(
\sigma _{i}\right) \right) $ to be the gradient flow for $i-th$ Hamiltonian
function on the potential surface, we get%
\begin{equation*}
\xi =\frac{dn}{d\sigma }=\partial _{n}H_{i}\text{, \ \ }\eta =\frac{db}{%
d\sigma }=\partial _{b}H_{i},\text{ \ \ \ \ }i=1,2.
\end{equation*}%
together with 
\begin{equation*}
\left\Vert \nabla H_{i}\right\Vert ^{2}=\left( \partial _{n}H_{i}\right)
^{2}+\left( \partial _{b}H_{i}\right) ^{2}=1
\end{equation*}%
Furthermore, from Eqs.(\ref{j1j2}) and (\ref{coco}) we can identify the
partial derivatives of Hamiltonian function and obtain%
\begin{equation}
\partial _{b}H_{i}=\mu _{i}\partial _{n}H_{i}  \label{e153}
\end{equation}%
or $\eta =\mu _{i}\xi $ and $\eta ^{2}+\xi ^{2}=1.$ Then, the gradient of
Hamiltonian functions on potential surface become two different
representations of the unit tangent vector of $\mathbf{X}\left( n\left(
\sigma \right) ,b\left( \sigma \right) \right) $ 
\begin{equation}
\mathbf{T}=\nabla H_{i}=\frac{1}{\sqrt{1+\mu _{i}^{2}}}\left( \mathbf{n}+\mu
_{i}\mathbf{b}\right) =\frac{\mathbf{J}_{i}}{\left\Vert \mathbf{J}%
_{i}\right\Vert },\qquad i=1,2.  \label{e154}
\end{equation}%
The integrability condition 
\begin{equation}
\partial _{n}\left( \partial _{b}H_{i}\right) -\partial _{b}\left( \partial
_{n}H_{i}\right) +\mathcal{H}_{tn}\partial _{n}H+\mathcal{H}_{tb}\partial
_{b}H=0  \label{e155}
\end{equation}%
of Eq.(\ref{e153}) 
\begin{equation}
\kappa _{g}=\frac{1}{\left( 1+\mu ^{2}\right) ^{\frac{3}{2}}}\left( \partial
_{n}\mu +\mu \partial _{b}\mu +\left( 1+\mu ^{2}\right) \left( \mathcal{H}%
_{tn}+\mu \mathcal{H}_{tb}\right) \right) =0  \label{e156}
\end{equation}%
is the vanishing geodesic curvature described in Eq.(\ref{gcurv}).

Note that the Riccati equation (\ref{ricca}) describing $\mu \left(
s,n,b\right) $ determines the partial derivative with respect to $s$
variable and allows an arbitrary dependence on $n$ and $b$ variables, while
the vanishing geodesic curvature condition depends only on the derivatives
with respect to $n$ and $b.$ Therefore, the choice of geodesic distances as
Hamiltonian functions puts a restriction on the arbitrariness admitted by
Eq.(\ref{ricca}).

\subsection{Local extension of Hamiltonian functions on potential surfaces}

On any potential surface, $F\left( \mathbf{x}\right) =c,$ we had two
gradient vectors $\nabla H_{1},\nabla H_{2}$ and the unit normal of the
surface $\mathbf{t}$. One can extend these Hamiltonian functions on
potential surfaces to functions on $%
\mathbb{R}
^{3}$ in infinitely many ways. Geometrically this amounts to embedding
geodesic curves to surfaces in $%
\mathbb{R}
^{3},$ which are the potential surfaces of Hamiltonian functions of
dynamical system (\ref{gradsys}) under consideration. These surfaces must
contain integral curves of Eq.(\ref{gradsys}). In other words, tangent plane
of surface obtained by extension of $\nabla H_{i}$ must contain all vectors
perpendicular to each $\nabla H_{i}$ in $%
\mathbb{R}
^{3},$ in particular, $\mathbf{t}$ and $\nabla H_{i}\times \mathbf{t.}$ In
fact, since $\mathbf{t}$ and $\nabla H_{i}\times \mathbf{t}$ are linearly
independent by definition, they span a two dimensional plane. The following
proposition proves that they always integrate to potential surface of the
corresponding Hamiltonian function. For convenience we will assume that Eq.(%
\ref{gradsys}) is of the form 
\begin{equation}
\mathbf{v}\left( \mathbf{x}\right) =\mathbf{\nabla }F\left( \mathbf{x}%
\right) =\mathbf{J}_{1}\left( \mathbf{x}\right) \times \nabla H_{2}\left( 
\mathbf{x}\right) =\mathbf{J}_{2}\left( \mathbf{x}\right) \times \nabla
H_{1}\left( \mathbf{x}\right)  \label{e157}
\end{equation}%
of a bi-Hamiltonian system, where%
\begin{equation}
\mathbf{J}_{1}=\phi \nabla H_{1}\text{, \ \ \ }\mathbf{J}_{2}=-\phi \nabla
H_{2}  \label{e158}
\end{equation}%
and the conformal factor is given by 
\begin{equation}
\phi =\frac{A_{1}A_{2}(\mu _{2}-\mu _{1})}{||\mathbf{v}||}.  \label{e158a}
\end{equation}

\begin{proposition}
\begin{equation}
\left[ \frac{1}{||\mathbf{v}||}\mathbf{t\cdot }\nabla \mathbf{,}\frac{1}{||%
\mathbf{v}||}\left( \mathbf{J}_{i}\times \mathbf{t}\right) \cdot \nabla %
\right] =\mathbf{0,}\qquad i=1,2  \label{holonomy}
\end{equation}
\end{proposition}

\begin{proof}
Let $\mathbf{w}=\xi \mathbf{n}+\eta \mathbf{b}$ be a vector field in the
tangent plane of a potential surface of the function $F.$ Then 
\begin{equation}
\begin{array}{lll}
\left[ \mathbf{t\cdot }\nabla \mathbf{,w\cdot }\nabla \right] & = & -%
\mathcal{H}_{bt}\xi \mathbf{t\cdot }\nabla \mathbf{+}\left( \partial _{s}\xi
-\mathcal{H}_{bn}\xi +\mathcal{H}_{n}\eta \right) \mathbf{n\cdot }\nabla \\ 
&  & \mathbf{+}\left( \partial _{s}\eta -\mathcal{H}_{b}\xi +\mathcal{H}%
_{nb}\eta \right) \mathbf{b\cdot }\nabla \mathbf{.}%
\end{array}
\label{bracket}
\end{equation}%
If $\mathbf{w}=\nabla H$ for some function $H$ on the potential surface, we
get%
\begin{equation}
\begin{array}{lll}
\left[ \mathbf{t\cdot }\nabla \mathbf{,}\nabla H\mathbf{\cdot }\nabla \right]
& = & -\mathcal{H}_{bt}\partial _{n}H\mathbf{t\cdot }\nabla \mathbf{-}\left( 
\mathbf{2}\mathcal{H}_{bn}\partial _{n}H+(\mathcal{H}_{b}\mathbf{-}\mathcal{H%
}_{n})\partial _{b}H\right) \mathbf{n\cdot }\nabla \\ 
&  & \mathbf{+}\left( 2\mathcal{H}_{nb}\partial _{b}H+\left( \mathcal{H}_{n}-%
\mathcal{H}_{b}\right) \partial _{n}H\right) \mathbf{b\cdot }\nabla \mathbf{.%
}%
\end{array}
\label{gradbracket}
\end{equation}%
Using (\ref{gradbracket}), it is easy to compute 
\begin{equation}
\left[ \mathbf{t\cdot }\nabla \mathbf{,}\left( \nabla H\times \mathbf{t}%
\right) \mathbf{\cdot }\nabla \right] =-\mathcal{H}_{bt}\partial _{b}H%
\mathbf{t\cdot }\nabla \mathbf{-}\left( \nabla \cdot \mathbf{t}\right)
\left( \nabla H\times \mathbf{t}\right) \mathbf{\cdot }\nabla
\label{ortgrad}
\end{equation}%
which is sufficient to show that the space spanned by $\mathbf{t}$ and $%
\nabla H\times \mathbf{t}$ is tangent to a surface in $%
\mathbb{R}
^{3}.$ Using the equations 
\begin{gather}
\partial _{s}\mu _{i}=\mathcal{H}_{n}+\mu _{i}(\mathcal{H}_{nb}+\mathcal{H}%
_{bn})+\mu _{i}^{2}\mathcal{H}_{b}\qquad i=1,2  \label{e159} \\
\text{ }\partial _{s}\ln A_{i}=\partial _{s}\ln \left\Vert \mathbf{v}%
\right\Vert -\mu _{i}\mathcal{H}_{b}-\mathcal{H}_{nb}\qquad \text{\ \ \ \ }%
i=1,2  \label{e160}
\end{gather}%
defining the Poisson structures we obtain 
\begin{equation}
\partial _{s}\ln \left( \mu _{2}-\mu _{1}\right) =(\mathcal{H}_{nb}+\mathcal{%
H}_{bn})+\left( \mu _{2}+\mu _{1}\right) \mathcal{H}_{b}  \label{e161}
\end{equation}%
\begin{equation}
\partial _{s}\ln \frac{A_{i}}{\left\Vert \mathbf{v}\right\Vert }=-\mu _{i}%
\mathcal{H}_{b}-\mathcal{H}_{nb}.  \label{e162}
\end{equation}%
Eqs.(\ref{e161}) and (\ref{e162}) when used in Eq.(\ref{e158a}) result in%
\begin{equation}
\partial _{s}\ln \frac{\phi }{||\mathbf{v}||}=\mathcal{H}_{bn}-\mathcal{H}%
_{nb}\text{ }\mathbf{=}\text{ }\nabla \cdot \mathbf{t}  \label{e163}
\end{equation}%
for the divergence of $\mathbf{t.}$ This leads, with Eq.(\ref{ortgrad}), to 
\begin{equation}
\left[ \mathbf{t\cdot }\nabla \mathbf{,}\frac{\phi }{||\mathbf{v}||}\left(
\nabla H_{i}\times \mathbf{t}\right) \mathbf{\cdot }\nabla \right] =-\frac{%
\phi }{||\mathbf{v}||}\mathcal{H}_{bt}\partial _{b}H_{i}\mathbf{t\cdot }%
\nabla  \label{e164}
\end{equation}%
By Eq.(\ref{e162}) we have 
\begin{equation}
\partial _{n}H_{1}=-\frac{A_{1}}{\phi },\text{ \ \ }\partial _{b}H_{1}=-%
\frac{A_{1}\mu _{1}}{\phi },\text{ \ \ }\partial _{n}H_{2}=\frac{A_{2}}{\phi 
},\text{ \ \ }\partial _{b}H_{1}=\frac{A_{2}\mu _{2}}{\phi }  \label{e165}
\end{equation}%
and we can rewrite (\ref{e164}) in the following form:%
\begin{equation}
\left[ \mathbf{t\cdot }\nabla \mathbf{,}\frac{A_{i}}{||\mathbf{v}||}\left(
\mu _{i}\mathbf{n}-\mathbf{b}\right) \mathbf{\cdot }\nabla \right] =-\frac{%
A_{i}\mu _{i}}{||\mathbf{v}||}\mathcal{H}_{bt}\mathbf{t\cdot }\nabla
\label{e184}
\end{equation}%
Now, for any vector field $\mathbf{u}$ satisfying 
\begin{equation}
\left[ \mathbf{t\cdot }\nabla \mathbf{,u\cdot }\nabla \right] =\lambda 
\mathbf{t\cdot }\nabla  \label{e185}
\end{equation}%
we can find an integrating factor $\alpha $ such that%
\begin{equation}
\left[ \alpha \mathbf{t\cdot }\nabla \mathbf{,u\cdot }\nabla \right] =\left(
\alpha \lambda -\mathbf{u\cdot }\nabla \alpha \right) \mathbf{t\cdot }\nabla 
\mathbf{=0}  \label{e186}
\end{equation}%
and therefore 
\begin{equation}
\alpha \lambda -\mathbf{u\cdot }\nabla \alpha =0  \label{e187}
\end{equation}%
Taking 
\begin{equation}
\mathbf{u=}\frac{A_{i}}{||\mathbf{v}||}\left( \mu _{i}\mathbf{n}-\mathbf{b}%
\right) \text{ and }\lambda =-\frac{A_{i}\mu _{i}}{||\mathbf{v}||}\mathcal{H}%
_{bt}  \label{e188}
\end{equation}%
Eq.(\ref{e187}) amounts to%
\begin{equation}
\partial _{b}\alpha -\mu _{i}\partial _{n}\alpha =\mu _{i}\mathcal{H}_{bt}%
\text{ \ \ }i=1,2.  \label{e189}
\end{equation}%
Since $\mu _{2}-\mu _{1}\neq 0$ we simply get%
\begin{equation}
\partial _{n}\alpha =-\mathcal{H}_{bt},\qquad \partial _{b}\alpha =0.
\label{e190}
\end{equation}%
To find a solution to (\ref{e190}), note that, for any function $\varphi
\left( s\right) ,$ the derivative $\partial _{s}\varphi \left( s\right) $
satisfies 
\begin{eqnarray}
\partial _{n}\partial _{s}\varphi \left( s\right) &=&\mathcal{H}%
_{bt}\partial _{s}\varphi \left( s\right)  \label{e191} \\
\partial _{b}\partial _{s}\varphi \left( s\right) &=&-\mathcal{H}%
_{nt}\partial _{s}\varphi \left( s\right) .  \label{e191a}
\end{eqnarray}%
According to Proposition 1 above, $\mathcal{H}_{nt}=0$ for the cases 1 and
2bii of Frenet-Serret frames that can be applied to gradient systems.
Therefore, we have%
\begin{eqnarray}
\partial _{n}\partial _{s}\varphi \left( s\right) &=&\mathcal{H}%
_{bt}\partial _{s}\varphi \left( s\right)  \label{e192} \\
\partial _{b}\partial _{s}\varphi \left( s\right) &=&0  \notag
\end{eqnarray}%
and choosing 
\begin{equation}
\alpha =\frac{1}{\partial _{s}\varphi \left( s\right) }  \label{e193}
\end{equation}%
solves Eq.(\ref{e187}). It is possible to get a more specific solution using
Eqs.(\ref{potsur2}) and (\ref{potsur}) which state that 
\begin{equation}
||\mathbf{v}||=\partial _{s}F\left( s\right) .  \label{e194}
\end{equation}%
Therefore one may choose 
\begin{equation}
\alpha =\frac{1}{||\mathbf{v}||}  \label{e195}
\end{equation}%
which proves the proposition
\end{proof}

Above proposition proves that the arclengths of integral curves of vector
fields 
\begin{equation*}
\frac{1}{||\mathbf{v}||}\mathbf{t},\text{ }\frac{1}{||\mathbf{v}||}\mathbf{J}%
_{i}\times \mathbf{t}
\end{equation*}%
provide parametrizations and hence a local coordinate system for the
potential surface $H_{i}=c.$

\section{Examples}

\subsection{Sphere}

Let $\mathbf{r}=\left( x,y,z\right) $ and $r=\left\Vert \mathbf{r}%
\right\Vert =\sqrt{x^{2}+y^{2}+z^{2}}.$ Consider the gradient system 
\begin{equation}
\frac{d\mathbf{r}\left( t\right) }{dt}=\mathbf{r}\left( t\right)  \label{ex1}
\end{equation}%
with the potential function 
\begin{equation}
F\left( \mathbf{r}\right) =\frac{r^{2}}{2}  \label{ex2}
\end{equation}%
where level surfaces are spheres. To construct the Frenet-Serret frame, we
begin with the unit tangent vector 
\begin{equation}
\mathbf{t}=\frac{\mathbf{r}}{r}  \label{ex3}
\end{equation}%
and take normal and binormal vectors as the lines of curvature of level
surfaces of the potential function. Since level surfaces are spheres, which
are surfaces of revolution, their lines of curvatures are latitudes and
longitudes. Adapting spherical coordinates%
\begin{equation}
\begin{array}{lll}
x & = & r\sin \phi \cos \theta \\ 
y & = & r\sin \phi \sin \theta \\ 
z & = & r\cos \phi%
\end{array}
\label{ex4}
\end{equation}%
the Frenet-Serret frame can be written as 
\begin{equation}
\begin{array}{lll}
\mathbf{t} & = & \sin \phi \cos \theta \mathbf{i}+\sin \phi \sin \theta 
\mathbf{j}+\cos \phi \mathbf{k} \\ 
\mathbf{n} & = & -\sin \theta \mathbf{i}+\cos \theta \mathbf{j} \\ 
\mathbf{b} & = & \cos \phi \cos \theta \mathbf{i}+\cos \phi \sin \theta 
\mathbf{j}-\sin \phi \mathbf{k.}%
\end{array}
\label{ex5}
\end{equation}%
For the construction of Hamiltonian functions, we will fix the unit sphere.
The distance function on the unit sphere is%
\begin{equation}
d\left( P_{i},P\right) =\arccos \left( \mathbf{P}_{i}\mathbf{\cdot P}\right)
\label{ex6}
\end{equation}%
where $\mathbf{P}$ and $\mathbf{P}_{i}$ are position vectors of the points $%
P $ and $P_{i}.$ Written explicitly, if 
\begin{eqnarray}
\mathbf{P}_{i} &=&\left( \sin \phi _{i}\cos \theta _{i},\sin \phi _{i}\sin
\theta _{i},\cos \phi _{i}\right)  \label{ex7a} \\
\mathbf{Q} &=&\left( \sin \phi \cos \theta ,\sin \phi \sin \theta ,\cos \phi
\right)  \label{ex7b}
\end{eqnarray}%
then%
\begin{equation}
d\left( P,P_{i}\right) =\arccos \left( \cos \phi \cos \phi _{i}+\sin \phi
\sin \phi _{i}\cos \left( \theta -\theta _{i}\right) \right) .  \label{ex8}
\end{equation}%
Now choosing the north pole $P_{1}=\left( 0,0,1\right) $, i.e. $\phi _{1}=0$
and $\theta _{1}=0,$ as our fixed point, we get 
\begin{equation}
d\left( P_{1},P\right) =\phi  \label{ex9}
\end{equation}%
and choosing the point on the equator $P_{2}=\left( 1,0,0\right) $, i.e. $%
\phi _{2}=\pi /2$ and $\theta _{2}=0,$ the distance becomes%
\begin{equation}
d\left( P_{2},P\right) =\arccos \left( \sin \phi \cos \theta \right) .
\label{ex10}
\end{equation}%
To extend these two functions to $%
\mathbb{R}
^{3}$, we define the Hamiltonian functions%
\begin{eqnarray}
H_{1}\left( \mathbf{r}\left( t\right) \right) &=&H_{1}\left( \frac{\mathbf{r}%
\left( t\right) }{r\left( t\right) }\right) =d\left( P_{1},\frac{\mathbf{r}%
\left( t\right) }{r\left( t\right) }\right) =\phi  \label{ex11a} \\
H_{2}\left( \mathbf{r}\left( t\right) \right) &=&H_{2}\left( \frac{\mathbf{r}%
\left( t\right) }{r\left( t\right) }\right) =d\left( P_{2},\frac{\mathbf{r}%
\left( t\right) }{r\left( t\right) }\right) =\arccos \left( \sin \phi \cos
\theta \right) .  \label{ex11b}
\end{eqnarray}%
These functions can be written in terms of $y/x$ and $z/x$ as%
\begin{eqnarray}
H_{1}\left( \mathbf{r}\left( t\right) \right) &=&\arccos \left( \frac{1}{%
\sqrt{\left( x/z\right) ^{2}+\left( y/x\right) ^{2}\left( x/z\right) ^{2}+1}}%
\right)  \label{ex12a} \\
H_{2}\left( \mathbf{r}\left( t\right) \right) &=&\arccos \left( \frac{x}{%
\sqrt{1+\left( y/x\right) ^{2}+\left( z/x\right) ^{2}}}\right) .
\label{ex12b}
\end{eqnarray}%
Indeed, the functions 
\begin{equation*}
h_{1}\left( x,y,z\right) =\frac{y}{x},\text{ }h_{2}\left( x,y,z\right) =%
\frac{z}{x}
\end{equation*}%
are the fundamental conserved quantities of Eq.(\ref{ex1}).

\subsection{ A Linear Poisson system}

Consider the gradient system defined by the vector field 
\begin{equation}
\mathbf{v}(\mathbf{x})=(yz,xz,xy)  \label{ex13}
\end{equation}%
which is an unphysical version of the Euler top. The potential surfaces are 
\begin{equation}
F\left( x,y,z\right) =xyz=c  \label{ex14}
\end{equation}%
As mentioned in \cite{wat}, the geodesic flows on potential surfaces for $%
c\neq 0$ are not integrable, and hence one may not expect to find simple
Hamiltonian functions for this system. However, letting $c=0$ yields 
\begin{equation}
xyz=0  \label{ex15}
\end{equation}%
which is nothing but the non-smooth union of coordinate planes. On this
surface consider the closed subset 
\begin{equation}
x\geq 0,\text{ }y\geq 0,\text{ }z=0  \label{ex16}
\end{equation}%
which is the first quadrant of $xy-$ plane. Then, the Frenet-Serret frame
becomes 
\begin{equation}
\left( \mathbf{t,n,b}\right) =\left( \mathbf{k,i,j}\right)  \label{ex17}
\end{equation}%
Choosing $P_{1}=\left( 0,0\right) $ and $P_{2}=\left( 1,0\right) $ on the $%
xy-$plane it is possible to define two distance functions%
\begin{eqnarray}
d\left( P_{1},P\right) &=&\sqrt{x^{2}+y^{2}}  \label{ex18} \\
d\left( P_{2},P\right) &=&\sqrt{\left( x-1\right) ^{2}+y^{2}}  \label{ex18a}
\end{eqnarray}%
whose gradients are%
\begin{eqnarray}
\nabla d\left( P_{1},P\right) &=&\left( \frac{x}{\sqrt{x^{2}+y^{2}}},\frac{y%
}{\sqrt{x^{2}+y^{2}}}\right)  \label{ex19} \\
\nabla d\left( P_{2},P\right) &=&\left( \frac{x-1}{\sqrt{\left( x-1\right)
^{2}+y^{2}}},\frac{y}{\sqrt{\left( x-1\right) ^{2}+y^{2}}}\right)
\label{ex19a}
\end{eqnarray}%
Using the immediate condition%
\begin{equation}
\frac{d}{dt}\left( x\left( t\right) y\left( t\right) z\left( t\right)
\right) =0  \label{ex20}
\end{equation}%
we obtain the time evolution of $z$ coordinate 
\begin{equation}
\frac{dz\left( t\right) }{dt}=-\frac{z\left( t\right) }{x\left( t\right) }%
\frac{dx\left( t\right) }{dt}-\frac{z\left( t\right) }{y\left( t\right) }%
\frac{dy\left( t\right) }{dt}.  \label{ex21}
\end{equation}%
That means, we can extend to the gradient vector fields in Eqs.(\ref{ex19})
and (\ref{ex20}) to the vector fields%
\begin{eqnarray}
\mathbf{u}_{1}\left( \mathbf{x}\right) &=&\frac{1}{\sqrt{x^{2}+y^{2}}}\left(
x,y,-2z\right)  \label{ex22} \\
\mathbf{u}_{2}\left( \mathbf{x}\right) &=&\frac{1}{\sqrt{\left( x-1\right)
^{2}+y^{2}}}\left( x-1,y,-2z+\frac{z}{x}\right)  \label{ex22a}
\end{eqnarray}%
which form a frame for the tangent space of the potential surface $xyz=0$.
Using the decomposition 
\begin{equation}
\mathbf{u}_{2}\left( \mathbf{x}\right) =\frac{1}{\sqrt{\left( x-1\right)
^{2}+y^{2}}}\left( x,y,-2z\right) +\frac{1}{x\sqrt{\left( x-1\right)
^{2}+y^{2}}}\left( -x,0,z\right)  \label{ex23}
\end{equation}%
it is possible find another frame consisting of gradient vector fields 
\begin{eqnarray}
\nabla H_{1} &=&\left( x,y,-2z\right)  \label{ex24} \\
\nabla H_{2} &=&\left( -x,0,z\right)  \label{ex24a}
\end{eqnarray}%
for the functions 
\begin{eqnarray}
H_{1}\left( x,y,z\right) &=&\frac{1}{2}\left( x^{2}+y^{2}-2z^{2}\right)
\label{ex24b} \\
H_{2}\left( x,y,z\right) &=&\frac{1}{2}\left( z^{2}-x^{2}\right)
\label{ex24c}
\end{eqnarray}%
which are Hamiltonian functions of the gradient system defined by Eq.(\ref%
{ex13}).

\subsection{The Aristotelian Model of Three Body Motion}

In \cite{ari}, we showed that the Aristotelian model of three-body motion is
defined by the gradient vector field 
\begin{equation}
\mathbf{v}(\mathbf{x})=(\frac{c}{x-y}+\frac{b}{x-z},\frac{a}{y-z}+\frac{c}{%
y-x},\frac{b}{z-x}+\frac{a}{z-y})  \label{ex25}
\end{equation}%
where the potential surfaces are 
\begin{equation}
F\left( \mathbf{x}\right) =a\ln \left( y-z\right) +b\ln \left( x-z\right)
+c\ln \left( x-y\right) =K.  \label{ex26}
\end{equation}%
We will choose the level surface defined by the constant 
\begin{equation}
K=a+b+c=1  \label{ex27}
\end{equation}%
and parametrized by 
\begin{equation}
u=\frac{x-z}{y-z}-\frac{1}{2},\text{ \ }v=y-z  \label{ex28}
\end{equation}%
so that the potential surface in new coordinates will be 
\begin{equation}
\ln v\left( u+\frac{1}{2}\right) ^{b}\left( u-\frac{1}{2}\right) ^{c}=1.
\label{ex29}
\end{equation}%
We can solve the variable $v$%
\begin{equation}
v=f\left( u\right) =\frac{e}{\left( u+\frac{1}{2}\right) ^{b}\left( u-\frac{1%
}{2}\right) ^{c}}  \label{ex30}
\end{equation}%
and obtain the relation 
\begin{equation*}
\frac{f^{\prime }\left( u\right) }{f\left( u\right) }=-\left( \frac{b}{u+%
\frac{1}{2}}+\frac{c}{u-\frac{1}{2}}\right) .
\end{equation*}%
With the parameters $\left( u,z\right) $, the potential surface becomes 
\begin{equation}
\mathbf{X}\left( u,z\right) =f\left( u\right) \left( u+\frac{1}{2}%
,1,0\right) +z\left( 1,1,1\right)  \label{ex31}
\end{equation}%
which is a ruled surface. To obtain the orthogonal parametrization, let 
\begin{equation}
\mathbf{e}_{1}=\frac{1}{\sqrt{3}}\left( 1,1,1\right) ,\text{ }\mathbf{e}_{2}=%
\frac{1}{\sqrt{6}}\left( 2,-1,-1\right) ,\text{ }\mathbf{e}_{3}=\frac{1}{%
\sqrt{2}}\left( 0,1,-1\right)  \label{ex32}
\end{equation}%
and define 
\begin{gather}
\mathbf{\alpha }\left( u\right) =\frac{\sqrt{2}uf\left( u\right) }{\sqrt{3}}%
\mathbf{e}_{2}+\frac{f\left( u\right) }{\sqrt{2}}\mathbf{e}_{3}
\label{ex34a} \\
w=\sqrt{3}z+\frac{1}{\sqrt{3}}f\left( u\right) \left( u+\frac{3}{2}\right)
\label{ex34b}
\end{gather}%
Then, we have%
\begin{equation}
\mathbf{X}\left( u,w\right) =\mathbf{\alpha }\left( u\right) +w\mathbf{e}%
_{1}.  \label{ex33}
\end{equation}%
The fundamental forms are 
\begin{eqnarray}
\left( g_{ij}\right) &=&\left( 
\begin{array}{cc}
\left\Vert \mathbf{\alpha }^{\prime }\left( u\right) \right\Vert ^{2} & 0 \\ 
0 & 1%
\end{array}%
\right)  \label{ex35} \\
\left( l_{ij}\right) &=&\frac{1}{\left\Vert \nabla F\right\Vert }\left( 
\begin{array}{cc}
\mathbf{\alpha }^{\prime \prime }\left( u\right) \cdot \nabla F & 0 \\ 
0 & 0%
\end{array}%
\right)  \label{ex36}
\end{eqnarray}%
where 
\begin{gather}
\nabla F=-\frac{\sqrt{3}f^{\prime }\left( u\right) }{\sqrt{2}f^{2}\left(
u\right) }\mathbf{e}_{2}+\frac{\sqrt{2}\left( f\left( u\right) +uf^{\prime
}\left( u\right) \right) }{f^{2}\left( u\right) }\mathbf{e}_{3}  \label{ex37}
\\
\mathbf{\alpha }^{\prime }\left( u\right) =\frac{\sqrt{2}\left( uf^{\prime
}\left( u\right) +f\left( u\right) \right) }{\sqrt{3}}\mathbf{e}_{2}+\frac{%
f^{\prime }\left( u\right) }{\sqrt{2}}\mathbf{e}_{3}.  \label{ex37a}
\end{gather}%
Note also that%
\begin{equation*}
\nabla F=\frac{\sqrt{3}}{f^{2}\left( u\right) }\mathbf{e}_{1}\times \mathbf{%
\alpha }^{\prime }\left( u\right) .
\end{equation*}%
Since the Gaussian curvature of this ruled surface is zero, the potential
surfaces are developable surfaces, and therefore can be mapped isometrically
onto the plane. We may choose the Frenet-Serret frame on this surface to be
the Darboux frame 
\begin{equation}
\left( \mathbf{t,n,b}\right) =\left( \frac{\nabla F}{\left\Vert \nabla
F\right\Vert },\mathbf{e}_{1},\frac{\mathbf{\alpha }^{\prime }\left(
u\right) }{\left\Vert \mathbf{\alpha }^{\prime }\left( u\right) \right\Vert }%
\right)  \label{e38}
\end{equation}%
of lines of curvature which are the coordinate curves $u=c_{1}$ and $%
w=c_{2}. $ Since $\nabla F,$ $\mathbf{\alpha }^{\prime }\left( u\right) $
and $\mathbf{\alpha }^{\prime \prime }\left( u\right) $ are vectors in the
plane spanned by $\mathbf{e}_{2}$ and $\mathbf{e}_{3},$ we have 
\begin{equation*}
\nabla F\cdot \left( \mathbf{\alpha }^{\prime }\left( u\right) \times 
\mathbf{\alpha }^{\prime \prime }\left( u\right) \right) =0
\end{equation*}%
and since $\mathbf{e}_{1}$ is a constant vector, it is easy to see that
these lines of curvature are also geodesics. Therefore, the coordinates $%
\left( s,n,b\right) $ are geodesic distances in $\mathbf{t,n,b}$ directions,
respectively. The metric in \ the new coordinate system becomes Euclidean.
This allows us to write 
\begin{equation}
\left( \mathbf{n,b}\right) =\left( \nabla H_{1}\left( n\right) ,\nabla
H_{2}\left( b\right) \right)  \label{ex39}
\end{equation}%
where $H_{1}$ and $H_{2}$ are geodesic distance functions in directions of $%
\mathbf{n}$ and $\mathbf{b.}$ In coordinates $\left( u,v,w\right) ,$ the
first Hamiltonian function is 
\begin{equation}
H_{1}\left( w\right) =\int_{w_{0}}^{w}\left\Vert \mathbf{e}_{1}\right\Vert
dt=w  \label{ex40}
\end{equation}%
and it is easy to see that 
\begin{equation}
w=\frac{1}{\sqrt{3}}\left( x+y+z\right) .  \label{ex41}
\end{equation}%
For the second Hamiltonian function, we have 
\begin{equation}
H_{2}\left( u\right) =\int_{u_{0}}^{\upsilon }\left\Vert \mathbf{\alpha }%
^{\prime }\left( t\right) \right\Vert dt=\int_{u_{0}}^{u}\sqrt{\frac{2}{3}%
\left( tf^{\prime }\left( t\right) +f\left( t\right) \right) ^{2}+\frac{1}{2}%
\left( f^{\prime }\left( t\right) \right) ^{2}}dt  \label{ex42}
\end{equation}%
or equivalently 
\begin{equation*}
H_{2}\left( x,y,z\right) =\sqrt{\frac{2}{3}}\int_{u_{0}}^{\frac{2x-y-z}{%
2\left( y-z\right) }}\sqrt{\left( t+\frac{f\left( t\right) }{f^{\prime
}\left( t\right) }\right) ^{2}+\frac{3}{4}}df\left( t\right) .
\end{equation*}

\section{Conclusion}

We developed differential calculus in Frenet-Serret frame. We extended the
result of \cite{eh09} for constructing Frenet-Serret frame to all dynamical
systems with the help of eigenvectors of curl operator and a result of
Chandrasekhar and Kendall in \cite{CK}. Considering bi-Hamiltonian structure
and Jacobi identity in Frenet-Serret frame associated to a dynamical system,
we proved that all dynamical systems in three dimensions possess two
compatible Poisson structures. We also presented the relation between
Hamiltonian functions and Poisson vectors.

Given a gradient dynamical system, we presented the geometric parameters of
both level surface and an arbitrary curve on it. In particular, we
considered, on level surfaces of potential function, gradient flows of
restrictions of Hamiltonian functions and proved that it is possible to find
Hamiltonian functions whose gradient flows on level surface have geodesic
curvature zero. This result led us to show that Hamiltonian functions are
determined by distance functions, namely, geodesic lenghts from an arbitrary
point to two different fixed points on the level surface of potential
function.

Finally, by means of transformations bringing one of the components of an
orthogonal metric to constant, we proved that finding two Hamiltonian
functions of a gradient system is the same as constructing geodesic
coordinates of its potential surfaces. As examples, we worked out decoupled
flow of radius vector of a sphere, a quadratic dynamical system possessing
linear Poisson structures, and the Aristotelian model of three body
motion.

\end{document}